\documentclass{amsart}

\usepackage{amsmath,amssymb,mathtools}
\usepackage{amsrefs}
\usepackage{xspace}
\usepackage{proof}
\usepackage{bussproofs}
\usepackage{url}
\usepackage{stmaryrd}
\usepackage{soul}
\usepackage{todonotes}

\usepackage{hyperref}
\hypersetup{pdfauthor={Name}}

\newtheorem{theorem}{Theorem}
\newtheorem{proposition}[theorem]{Proposition}
\newtheorem{lemma}[theorem]{Lemma}

\newtheorem{corollary}[theorem]{Corollary}
\newtheorem{claim}[theorem]{Claim} 
\theoremstyle{definition}
\newtheorem{definition}[theorem]{Definition}

\theoremstyle{remark}
\newtheorem*{remark}{Remark}

\newcommand{\DomB}{\mathbb{B}}
\newcommand{\Dom}{\mathbb{D}}

\newcommand{\APPROXEQ}{\unlhd}
\newcommand{\PET}{\ensuremath{\mathbf{PET}}\xspace}
\newcommand{\PETS}{\ensuremath{\mathbf{PETS}}\xspace}

\newcommand{\Ax}{\ensuremath{\mathrm{Ax}}\xspace}
\newcommand{\Seq}[1]{\langle {#1} \rangle}

\newcommand{\Finite}[1]{\tilde{{#1}}}

\newcommand{\D}{\mathcal{D}}

\newcommand{\Inst}{\mathrm{Inst}}
\newcommand{\InstR}{\overrightarrow{\Inst}}
\newcommand{\InstL}{\overleftarrow{\Inst}}
\newcommand{\IAx}{\mathrm{A}}

\newcommand{\ISup}{\mathrm{S}{\uparrow}}
\newcommand{\ISdwn}{\mathrm{S}{\downarrow}}

\DeclareMathOperator{\var}{\mathrm{Var}}
\DeclareMathOperator{\bvar}{\mathrm{BVar}}

\DeclareMathOperator{\dom}{\mathrm{dom}}
\DeclareMathOperator{\rng}{\mathrm{rng}}

\DeclareMathOperator{\ar}{\mathrm{ar}}

\newcommand{\cB}{\mathcal{B}}
\newcommand{\cD}{\mathcal{D}}
\newcommand{\cF}{\mathcal{F}}
\newcommand{\cT}{\mathcal{T}}
\newcommand{\cX}{\mathcal{X}}

\newcommand{\MANY}[1]{\overline{{#1}}}
\newcommand{\vs}{\MANY{s}}
\newcommand{\vt}{\MANY{t}}
\newcommand{\vu}{\MANY{u}}
\newcommand{\vv}{\MANY{v}}
\newcommand{\vw}{\MANY{w}}
\newcommand{\vx}{\MANY{x}}
\newcommand{\vy}{\MANY{y}}

\newcommand{\eps}{\epsilon}
\newcommand{\si}{\sigma}
\newcommand{\al}{\alpha}

\newcommand{\Stheory}[1]{\ensuremath{\mathrm{S}^{#1}_2}\xspace}
\newcommand{\Ttheory}[1]{\ensuremath{\mathrm{T}^{#1}_2}\xspace}
\newcommand{\Sot}{\Stheory{1}}
\newcommand{\Stt}{\Stheory{2}}
\newcommand{\sib}[1]{\ensuremath{\Sigma^{\mathrm{b}}_{#1}}\xspace}
\newcommand{\pib}[1]{\ensuremath{\Pi^{\mathrm{b}}_{#1}}\xspace}
\DeclareMathOperator{\Oh}{\mathrm{O}} 
\newcommand{\sma}{\mathbin{\#}}
\newcommand{\ind}[1]{\ensuremath{#1\mathrm{-IND}}\xspace}
\newcommand{\lind}[1]{\ensuremath{#1\mathrm{-LIND}}\xspace}
\newcommand{\sip}[1]{\ensuremath{\Sigma^{\mathrm{p}}_{#1}}\xspace}
\newcommand{\pip}[1]{\ensuremath{\Pi^{\mathrm{p}}_{#1}}\xspace}

\newcommand{\W}{\mathbf{w}}

\DeclareMathOperator{\lh}{\mathbf{l}} 
\DeclareMathOperator{\seqlh}{\mathbf{sql}} 
\DeclareMathOperator{\G}{\mathbf{g}} 
\DeclareMathOperator{\E}{\mathbf{e}} 
\newcommand{\apprx}{\sqsubseteq}
\newcommand{\compat}{\mathbin{\triangle}}
\newcommand{\forder}{\sqsubseteq}
\DeclareMathOperator{\maxapprx}{\mbox{$\max_\sqsubseteq$}}

\DeclareMathOperator{\suc}{\mathrm{s}}

\newcommand{\Basis}{\cB}
\newcommand{\eval}[1]{\llbracket {#1} \rrbracket}
\newcommand{\evalr}[2]{\eval{#1}_{#2,\rho}}
\newcommand{\evalfr}[1]{\eval{#1}_{F,\rho}}
\newcommand{\evalfrp}[1]{\eval{#1}_{F',\rho'}}
\newcommand{\update}[3]{#1\colon #2\mapsto #3}
\newcommand{\updfvw}{\update{f}{\vv}{w}}
\newcommand{\updcons}{\mathbin{\,*\,}}
\newcommand{\seqadd}{\mathbin{\,{:}\,}}
\newcommand{\seqcons}{\mathbin{\,{:}{:}\,}}

\newcommand{\Fupdate}[4]{#1\;*\;\update{#2}{#3}{#4}}
\newcommand{\FupdFfvw}{\Fupdate{F}{f}{\vv}{w}}

\newcommand{\Cons}{\mathrm{Cons}}

\title[On proving consistency of equational theories]{On proving consistency of equational theories in Bounded Arithmetic}

\author{Arnold Beckmann and Yoriyuki Yamagata}

\begin{document}

\maketitle

\begin{center}
  \today
\end{center}

\begin{abstract}
  We consider pure equational theories that allow substitution
  but disallow induction, which we denote as \PETS,
  based on recursive definition of their function symbols.
  We show that the Bounded Arithmetic theory \Sot proves the
  consistency of \PETS.
  Our approach employs models for \PETS based on approximate values
  resembling notions from domain theory in Bounded Arithmetic,
  which may be of independent interest. 
\end{abstract}

\section{Introduction}

The question whether the hierarchy of Bounded Arithmetic theories 
is strict or not, is an important open problem 
due to its connections to corresponding questions 
about levels of the Polynomial Time Hierarchy~\cite{Buss:1986}.
One obvious route to address this problem is to make use of 
G\"odel's 2nd Incompleteness Theorem, 
using statements expressing consistency for Bounded Arithmetic theories.
Early lines of research aimed to restrict the formulation 
of consistency suitably to achieve this aim \cites{Buss:1986, Pudlak:1990}.

One particular programme is to consider consistency statements 
based on equational theories.
Buss and Ignjatovi{\'c}~\cite{Buss:Ignjatovic:1995} have shown 
that the consistency of an induction free version 
of Cook's equational theory PV \cite{Cook:1975} is not provable in \Sot, 
where \Sot is the Bounded Arithmetic theory related to polynomial time reasoning.
Their version of induction free PV is formulated in a system that allows, 
in addition to equations, also inequalities between terms, and Boolean formulas.
Furthermore, a number of properties have been stated as axioms.

On the other hand, in a pure equational setting, where lines in derivations 
are equations between terms, axioms are restricted 
to recursive definition of function symbols, and induction is not allowed, 
consistency becomes provable in Bounded Arithmetic: 
The first author has shown in~\cite{Beckmann:2002} 
that the consistency of pure equational theories, 
in which substitution is not allowed, 
is provable in~\Sot\ -- in particular this result applies to 
Cook's PV \cite{Cook:1975} without substitution and induction.
The second author of this paper has improved on this result 
in~\cite{Yamagata:2018} by showing that the consistency of PV 
without induction but with substitution is provable in \Stt, 
the second level of the hierarchy of Bounded Arithmetic theories.

In this paper, we extend both our previous results \cites{Beckmann:2002, Yamagata:2018}.
With $\PETS(\Ax)$ we denote the pure equational theory 
with substitution but without induction, 
based on some nice set of axioms $\Ax$ --
Cook's original PV  \cite{Cook:1975} without induction but with substitution is one example of such a theory.
The main aim of this paper is to show that the consistency of $\PETS(\Ax)$ 
is provable in \Sot, thus improving on both \cites{Beckmann:2002, Yamagata:2018}.
To this end we employ a novel method of defining models in Bounded Arithmetic based 
on approximate values, which may be of independent interest.
Our approach resembles elements from domain theory, however we leave a full 
treatment of domain theory in Bounded Arithmetic to future research.

%
%

In the next section, we briefly introduce Bounded Arithmetic 
and fix some notions used throughout the paper.
In  Section~\ref{sec:EqTh} we define pure equational theories $\PETS(\Ax)$, 
which will be more general than PV without induction 
in that arbitrary recursive functions may be considered.
This is followed in Section~\ref{sec:Semantics} by an introduction 
of approximate values and semantics based on approximation, 
leading to feasible evaluations of terms based on such approximation semantics.
Section~\ref{sec:FrameModels} then defines models for equational theories 
based on approximation semantics, including a suitably restricted version 
which can be expressed as a bounded formula and used in induction arguments 
inside Bounded Arithmetic.  
A key notion will be a way of updating such models with further information 
about approximate values of functions, in a way that preserves the notion 
of being a model, provably in~\Sot.
In Section~\ref{sec:SoundnessInS22} we prove our first main theorem 
showing a form of soundness for $\PETS(\Ax)$ based on 
approximation semantics in \Stt -- an immediate consequence will be 
that \Stt proves the consistency of $\PETS(\Ax)$.
The final sections improve on this approach to obtain proofs in \Sot:
In Section~\ref{sec:instructions} we introduce instructions 
that allow to encode sequences of terms and operations on them, 
which are extracted from derivations in $\PETS(\Ax)$.
In this way we obtain an explicit way of describing model constructions 
related to $\PETS(\Ax)$ derivations, which allows us to reduce 
the bounded quantifier complexity of induction assertion in the proof 
of our second main theorem in Section~\ref{sec:SoundnessInS12} 
to show an improved soundness property for $\PETS(\Ax)$ provable in \Sot.

\section{Bounded Arithmetic}

%
%

\subsection{Language of Bounded Arithmetic}

We give a brief introduction to Bounded Arithmetic to support the
developments in this paper.
For more in depth discussions and results we refer the interested
reader to the literature~\cites{Buss:1986,Krajicek:1995}.
Theories of Bounded Arithmetic are first order theories of
arithmetic similar to Peano Arithmetic, their domain of discourse are
the non-negative integers.
For the purpose of this paper we can assume that the language of
Bounded Arithmetic 
contains a symbol for each polynomial time computable function,
including
$0$, $1$, $+$, $\cdot$, $|.|$, $\sma$,
representing zero, one, addition, multiplication,
the binary length function $|x|$ that computes the number of bits in a
binary representation of $x$ and can be defined by
$|x|=\lceil\log_2(x+1)\rceil$,
and smash $\sma$ computing $x\sma y=2^{|x|\cdot|y|}$.

\subsection{Theories of Bounded Arithmetic}

Theories of Bounded Arithmetic contain suitable defining axioms
for its function symbols.
The main differentiator are various forms of induction for
various classes of bounded formulas, which we will define next.

\emph{Bounded quantifiers} are defined as follows:
\begin{align*}
  (\forall x\le t)A &\quad\text{abbreviates}\quad (\forall x) (x\le t\rightarrow A) \\
  (\exists x\le t)A &\quad\text{abbreviates}\quad (\exists x) (x\le t\wedge A)
\end{align*}
If the bounding term $t$ of a bounded quantifier is of the form
$|t'|$, then the quantifier is called \emph{sharply bounded}.

Classes of bounded formulas $\sib i$ and $\pib i$ are defined in \cite{Buss:1986} 
by essentially counting alternations between existential and universal bounded quantifiers. 
Predicates defined by $\sib i$ and $\pib i$ formulas define
computational problems in corresponding classes of the Polynomial
Time Hierarchy of decision problems $\sip i$ and $\pip i$, respectively.
For example, those defined by $\sib1$ correspond exactly to NP.

In particular, the $\sib i$ and $\pib i$ classes include the following formulas:

\begin{itemize}
\item
  $\sib0=\pib0$ is the class of formulas build from atomic formulas and
  closed under Boolean connectives and sharply bounded quantification.

\item
  $\sib{i+1}$ includes all formulas of the form
  $(\exists x\le t)A$ with $A\in\pib i$.

\item
  $\pib{i+1}$ includes all formulas of the form
  $(\forall x\le t)A$ with $A\in\sib i$.
\end{itemize}

The theories $\Stheory{i}$, $i\ge0$, of Bounded Arithmetic have been defined 
in \cite{Buss:1986}, 
establishing a close relationship between fragments of Bounded Arithmetic
and levels of the Polynomial Time Hierarchy of functions.
More precisely, the $\sib{i+1}$-definable functions of
$\Stheory{i+1}$, that is the functions whose graph can be described by
a $\sib{i+1}$ formula, and whose totality is provable in
$\Stheory{i+1}$, form exactly the $i{+}1$-st level of
the Polynomial Time Hierarchy of functions,
$\mathrm{FP}^{\sip i}$.

Instead of defining the theory $\Stheory{i}$, we state some characteristic properties about induction provable in them.
%
Let \ind{\sib i} be the schema of induction for $\sib i$-properties,
consisting of formulas of the form
\[
  A(0) \wedge (\forall x)(A(x)\rightarrow A(x+1))
  \rightarrow (\forall x)A(x)
\]
for $A\in\sib i$.
The schema of \emph{logarithmic induction} \lind{\sib i} is then obtained by restricting the conclusion of induction to logarithmic values, that is
\[
  A(0) \wedge (\forall x)(A(x)\rightarrow A(x+1))
  \rightarrow (\forall x)A(|x|)
\]
for $A\in\sib i$.
%
%
%
We have the following:

\begin{theorem}[\cite{Buss:1986}]\label{thm:Pi-LIND}
  The instances of $\lind{\sib i}$ and $\lind{\pib i}$ are provable in $\Stheory{i}$.
\end{theorem}

We already mentioned the intricate relationship between Bounded Arithmetic theories and
the Polynomial Time Hierarchy in terms of definable functions. 
Furthermore, it is know that a collapse of the hierarchy of Bounded Arithmetic theories is equivalent to a collapse of the Polynomial Time Hierarchy, provable in Bounded Arithmetic~\cites{Krajicek:Pudlak:Takeuti:1991,Buss:1995,Zambella:1996};
With $\Ttheory{i}$ denoting the theory $\ind{\sib i}$ we have that
$\Ttheory{i}=\Stheory{i+1}$ is equivalent to $\sip{i+1}\subseteq\pip{i+1}/\mathit{poly}$ provable in $\Ttheory{i}$.

The Bounded Arithmetic theory \Sot is able to formalize meta-mathematics and essential constructions to prove G{\"o}del's Incompleteness Theorems~\cite{Buss:1986}.
A basis for such formalization is a feasible encoding of sequences of numbers.
For this paper we assume that a suitable encoding of sequences and operations on them can be formalized in $\Sot$ as done in~\cite{Buss:1986}. 
We assume the following notation:

\begin{itemize}
\item 
  With $\Seq{x_1, \dots, x_k}$ we denote the encoding of sequence
  $x_1, \dots, x_k$.
  We will use $\si$, $\tau$ etc to range over sequence encodings.
\item
  With `$\seqcons$' we denote concatenation of sequences:
  \[
    \Seq{x_1, \dots, x_k} \seqcons \Seq{x_{k+1}, \dots, x_{k+\ell}}
    \quad=\quad
    \Seq{x_{1}, \dots, x_k, x_{k+1}, \dots, x_{k+\ell}}
  \]
\item
  With `$\seqadd$' we denote the function that adds an element to the left or
  right of a sequence:
  \begin{align*}
    x \seqadd \si  &\quad=\quad \Seq{x} \seqcons \si  \\
    \si \seqadd x  &\quad=\quad \si \seqcons \Seq{x}
  \end{align*}
\end{itemize}
The predicate `being a sequence encoding', as well as the operations `$\seqadd$' and `$\seqcons$', 
can be defined in $\Sot$ with their usual properties proven.

In the following we will concentrate on bounding the number of symbols in transformations and constructions.
For an object $o$ we will define its length $\lh(o)$ to be the number of symbols occurring in $o$.
As all our constructions will happen in the context of a given derivation $\D$, we obtain that
the size of the G{\"o}delization of object $o$ can then be bounded by $\lh(o)$ times the size of the G{\"o}delization of~$\D$.  

Furthermore, the constructions for defining $o$ in the context of $\D$ will always be explicit and simple enough to be formalizable in \Sot, similar to constructions in \cite{Buss:1986} dealing with meta-mathematical notions.
In those cases where more involved induction is needed (as in Theorems~\ref{thm:soundness in S22} and~\ref{thm:soundness in S12}) these will be analyzed carefully.



\subsection{Notations}

In the remaining part of this section we will fix some notation used throughout this paper.
We use $\equiv$ for equality of syntax.

\begin{itemize}
\item With $\# S$ we denote the cardinality of a set $S$.
\item For a function $f$, $\dom(f)$ denotes its domain,
  $\rng(f)$ its range.
\item We will use the abbreviation $\MANY{x}$
  for a sequence $x_1, \dots, x_n$ of objects.
\item For a set $X$, tuples in $X^n$ are denoted with
  $(x_1,\dots,x_n)$.
\item $\max(X)$ computes the maximum (according to a given
  order) of the elements in $X$.
  $\max$ is applied to a tuple by computing the maximal component
  in it.
\end{itemize}


\paragraph{Tuples and sequences}
Technically, there is a difference between a tuple of the form
$(s_1,\dots,s_n)$, which is an element of $S^n$,
and the sequence $s_1,\dots,s_n$, which is a formal list used e.g.\ as
the arguments to an $n$-ary function.
We will identify both and write $s\in S^n$ and $f(s)$ in the same
context, as long as it does not lead to confusion, in which case we
will choose a more precise differentiation.



\section{Equational Theories}\label{sec:EqTh}

\subsection{Domain of discourse}

The intended domain of discourse  $\DomB$ will be binary strings representing numbers.
$\DomB$ can be defined inductively as follows:
\begin{equation*}
        v ::= \epsilon \mid v0 \mid v1
\end{equation*}

We will also use terms denoting binary strings, which are formed from
constant~$\epsilon$ using unary function symbols $\suc_0$ and $\suc_1$
to add a single digit to the right of a string. 
We also use $t0$ to denote $\suc_0(t)$, and $t1$ for $\suc_1(t)$
for terms~$t$.

\begin{remark}
  Our results are not restricted to binary strings, but can be applied to general free algebras as domains of discourse as done in~\cite{Beckmann:2002}.
  However, for sake of simplicity we will only consider binary strings in this paper.
\end{remark}

\subsection{Terms}

We fix the language we use for equational theories.

\begin{definition}[Language for Equational Theories]
  We have the following basic ingredients:
  \begin{itemize}
  \item A countable set $\cX$ of \emph{variables;}
    we use $x, y, x_1, x_2, \ldots$ to denote variables.
  \item A countable set $\cF$ of \emph{function symbols;}
    we use $f, g, h,f_1, f_2, \ldots$ to denote function symbols. 
    Each function $f\in\cF$ comes with a non-negative integer $\ar(f)$
    called its \emph{arity}.
    We assume that $\epsilon$, $\suc_0$ and $\suc_1$ are included in
    $\cF$; 
    $\epsilon$, $\suc_0$ and $\suc_1$ form the set $\cB$ of \emph{basic function symbols.}
  \end{itemize}
\end{definition}

\begin{remark}  
  A function with arity $0$ is called a \emph{constant}.
  For example, $\epsilon$ is a constant.
\end{remark}

\begin{definition}[Terms]
  Let $X\subseteq\cX$ and $F\subseteq\cF$.
  The set $\cT(X,F)$ of \emph{terms over $F$ and $X$}, or simply
  \emph{terms}, is defined inductively as follows:
  \begin{itemize}
  \item All variables in $X$ are terms.
  \item If $f\in F$ has arity $n$ and $t_1, \ldots, t_n$ are terms,
    then $f(t_1, \ldots, t_n)$ is a term.
  \end{itemize}
  We will use $s,t,u$ to denote terms.
  
  The \emph{length $\lh(t)$ of term $t$} is defined in the
  following way:
  The length of a variable is $1$, and, recursively,
  \[  \lh(f(t_1, \dots, t_n)) = \lh(t_1)+\dots+\lh(t_n)+1 \enspace. \]
  With $\var(t)$ we denote the \emph{set of variables} that are
  occurring in a term~$t$.
\end{definition}


\begin{definition}[Substitution]
  Let $t,u$ be terms and $x$ be a variable.
  The \emph{substitution of $u$ for $x$ in $t$}, denoted $t[u/x]$,
  is obtained by replacing any occurrence of $x$ in $t$ by~$u$.
\end{definition}

We extend substitution to sequences of variables and terms of the same
length by successively substituting terms:
$t[\vu/\vx]$ stands for  $t[u_1/x_1][u_2/x_2]\dots[u_n/x_n]$.

\begin{definition}[Instance]
  A \emph{substitution instance}, or \emph{instance}, of an equation $s=t$ is any
  $s[\vu/\vx] = t[\vu/\vx]$ for sequences of variables $\vx$ and terms
  $\vu$ of the same length.
\end{definition}

\subsection{Nice axiom systems}

We will consider axioms consisting of equations that satisfy
particular conditions, which have been called \emph{nice}
in~\cite{Beckmann:2002}.

\begin{definition}[Axioms]
  Let \Ax be a \emph{set of nice axioms for $\cF$.}
  That is, each equation in \Ax
  must be of one of the following forms, for some
  $f\in\cF\setminus\cB$, some $t, t_\eps\in\cT(\cF,\{\vx\})$, and
  $t_0, t_1\in\cT(\cF,\{x,\vx\})$:
  \begin{align*}
    f(\vx) &= t \\
    f(\eps, \vx) &= t_\epsilon \\
    f(x0, \vx) &= t_0 \\
    f(x1, \vx) &= t_1. 
  \end{align*}
  Furthermore, the left-hand side of an equation is occurring at most
  once among equations in \Ax, also modulo substitution.
\end{definition}

\begin{remark}
  The definition implies that for any $t=u$ in \Ax we have
  $\var(u)\subseteq\var(t)$.
\end{remark}

\begin{remark}
  Consider a term $f(\vt)$ with $f\in\cF\setminus\cB$.
  The property of \Ax being nice implies that there is at most one axiom $t=u$ in \Ax 
  such that $f(\vt)$ can be written as a substitution instance of~$t$.
\end{remark}

\begin{remark}
  The definition of a nice axiom system in~\cite{Beckmann:2002} also contains a completeness
  condition, requiring that each function symbol in $\cF\setminus\cB$ is
  defined by an equation, and that the case distinction in the latter
  part is complete.
  We omit this form of completeness, as it is not needed for our developments.
\end{remark}


The left-hand side of an equation in \Ax is of a very special form: an
argument to the out-most function symbol can either be a variable,
$\eps$, or $\suc_i(x)$ for some variable~$x$.
We capture these forms in the following definition.

\begin{definition}[Generalized variable]\label{def:generalized var}
  A \emph{generalized variable} is a term which either is a variable,
  or $\eps$, or of the form $\suc_i(x)$ for some variable $x$.
\end{definition}

\begin{remark}
  Consider an axiom $t=u$ in \Ax.
  It follows that $t$ must be of the form $f(\vt)$,
  that each $t_i$ is a generalized variable, hence each $t_i$ contains at most one variable.
  Furthermore, the same variable cannot occur simultaneously in $t_i$ and $t_j$ for $i\neq j$.
\end{remark}

\begin{definition}[Rules for equational reasoning]
  Let $s,t,u,s_1,\dots,s_n,t_1,\dots,t_n$ be terms.
  The following are the rules that can be used to derive equations:
  \begin{description}
  \item[Axiom]
    $\vdash t=u$,
    where $t=u$ is an instance of an equation in~\Ax.
  \item[Reflexivity]
    $\vdash t=t$
  \item[Symmetry]
    $t=u \vdash u=t$
  \item[Transitivity]
    $t=s,s=u \vdash t=u$
  \item[Compatibility]
    $t=u \vdash s[t/x] = s[u/x]$
  \item[Substitution]
    $t=u \vdash t[s/x] = u[s/x]$.
  \end{description}
  In the case of Substitution as stated above, we say that the application of Substitution binds the variable $x$.
\end{definition}

\begin{remark}
  We also make use of a display style for presenting rules, like
  \begin{prooftree}
    \AxiomC{}
    \UnaryInfC{$t = t$}
  \end{prooftree}
  for Reflexivity Rule, or 
  \begin{prooftree}
    \AxiomC{$t = s$}
    \AxiomC{$s = u$}
    \BinaryInfC{$t = u$}
  \end{prooftree}
  for Transitivity Rule.
\end{remark}

\begin{definition}[Derivations]\label{def:derivations}
  A \emph{derivation} is a finite tree whose nodes are labelled with
  instances of rules for equational reasoning, such that for each
  node, the premises of the rule at that node coincide with the
  conclusions of rules at corresponding child nodes.

  Derivations can also be defined inductively from rules for
  equational reasoning:
  Any instance of an Axiom or Reflexivity Rule is a derivation ending
  in the equation given by that rule.
  If $R$ is a unary rule (like Symmetry, Compatibility or
  Substitution) with premise $e'$ and conclusion $e$,
  and $\cD'$ a derivation ending in $e'$, then
  \begin{prooftree}
    \AxiomC{$\cD'$}
    \noLine
    \UnaryInfC{$e'$}
    \LeftLabel{$R$}
    \UnaryInfC{$e$}
  \end{prooftree}
  is a derivation ending in $e$.
  The only binary rule we are considering is the Transitivity Rule.
  If $\cD_1$ a derivation ending in $t=s$, and
  $\cD_2$ a derivation ending in $s=u$,
  then
  \begin{prooftree}
    \AxiomC{$\cD_1$}
    \noLine
    \UnaryInfC{$t=s$}
    \AxiomC{$\cD_2$}
    \noLine
    \UnaryInfC{$s=u$}
    \LeftLabel{Transitivity}
    \BinaryInfC{$t=u$}
  \end{prooftree}
  is a derivation ending in $t=u$.

  The \emph{length $\lh(\cD)$ of a derivation $\cD$} is defined as the
  sum of the lengths of the equations occurring in it, plus the length of any additional syntax
  needed to identify applications of rules --- 
  for example, for an application of Substitution 
  $t=u \vdash t[s/x] = u[s/x]$
  we add $\lh(t,s,x,u,s,x)$ to avoid
  complications in cases where $x$ is not occurring in $t$ or $u$.
  The length $\lh(t=u)$ of an equation $t=u$ is set to $\lh(t)+\lh(u)+1$.
  With $\var(\cD)$ we denote the set of variables occurring in $\cD$.
  $\bvar(\cD)$ is the set of variables occurring in $\cD$ that
  are bound by an application of the Substitution Rule.
\end{definition}

\begin{definition}[Pure Equational Theories]
The \emph{pure equational theory} $\PET(\Ax)$ consists of all equations that can
be derived using the Axiom, Reflexivity, Symmetry, Transitivity and
Compatibility Rules.
The \emph{pure equational theory with substitution} $\PETS(\Ax)$ is
obtained by additionally allowing the Substitution Rule in the derivation of
equations.
\end{definition}




An instance $s[\vu/\vx]=t[\vu/\vx]$ of $s=t$ in \Ax is called an \emph{injective renaming} of $s=t$, iff
the variables $\vx$ are pairwise distinct,
they satisfy $\{\vx\}=\var(s,t)$,
and $\vu$ is a list of pairwise distinct variables.

\begin{proposition}[Variable Normal Form]\label{prop:VNF}
  For $\PETS(\Ax)$ derivations, we can assume the following normal
  form:
  \begin{enumerate}
  \item\label{prop:VNF1}
    Axiom $\vdash t=u$ only occur in the form where
    $t=u$ is obtained by injectively renaming variables of an equation
    in \Ax.
    This implies that $t$ is of the form $f(\vt)$ with each $t_i$ a
    generalized variable, and that the same variable is not occurring in both
    $t_i$ and $t_j$ for $i\neq j$.
  \item\label{prop:VNF2}
    Each variable occurring in a normal derivation is either occurring
    in the equation in which the derivation ends, or is removed in
    exactly one application of Substitution (as the variable which is bound by that application of Substitution).
  \end{enumerate}
  Furthermore, if $\D\vdash t=u$, then there exists $\D'$ in Variable Normal Form 
  such that $\D'\vdash t=u$ and $\lh(\D')=\Oh(\lh(\D)^2)$.
\end{proposition}

\begin{proof}
  For \eqref{prop:VNF1}, consider an equation $t=u$ in \Ax.
  As \Ax is nice, we have that $t$ is of the form $f(t_1,\dots,t_n)$
  with each $t_i$ a generalized variable potentially containing one
  variable $x_i$, and that these variables are pairwise distinct.
  Consider terms $s_1,\dots, s_n$, and the instance
  $\vdash t[\vs/\vx]=u[\vs/\vx]$ of the Axiom Rule.  This can be
  replaced by $\vdash t[\vy/\vx]=u[\vy/\vx]$ where $\vy$ is a list
  $y_1\dots,y_n$ of fresh, pairwise distinct variables, followed by
  applications of the Substitution Rule for successively applying
  substitutions $[s_1/y_1]$, $[s_2/y_2]$, \dots, $[s_n/y_n]$.

  For \eqref{prop:VNF2}, we observe that we can replace all
  occurrences of a fixed variable by a fresh variable throughout a
  derivation ending in an equation $e$, obtaining a similar derivation of the
  equation $e$ potentially with the former variable being renamed to the
  latter fresh variable.
  
  The above transformation at most squares the length of a derivation. 
\end{proof}

%
%
%

\begin{definition}[Formal Consistency]
  With $\Cons(\PETS(\Ax))$ we denote the sentence in the language of
  Bounded Arithmetic which expresses that there
  is no derivation in $\PETS(\Ax)$ ending in $0=1$, where
  $0$ denotes $\suc_0(\eps)$ and $1$ denotes $\suc_1(\eps)$.
\end{definition}

\section{Approximation Semantics}\label{sec:Semantics}

Infeasible values, although finite, can be considered as infinite bit
strings within theories of feasibility like Bounded Arithmetic.
This is relevant when evaluating functions formally in Bounded
Arithmetic.
Thus we will make use of notions from domain theory to define a method
for evaluating terms occurring in equational proofs.

\subsection{Approximate values}

The notion of approximate values $v$ is defined
in~\cite{Yamagata:2018},
which adds `unknown value' of a term~\cite{Beckmann:2002}, denoted with
`$*$', to bit-strings.

\begin{definition}[Approximate values]
  The set $\Dom$ of \emph{approximate values} is defined
  inductively as follows:
  \begin{equation*}
    v ::= \epsilon \mid v0 \mid v1 \mid *
  \end{equation*}
  The \emph{gauge} $\G(v)$ of $v\in\Dom$ is defined recursively:
    \begin{gather*}
        \G(\epsilon) = \G(*) = 1\\
        \G(v0) = \G(v1) = \G(v)+1
    \end{gather*}
    For a tuple $\vw = (w_1, \ldots, w_n) \in \Dom^n$,
    its \emph{gauge} is given as
    $\G(\vw) = \max\{\G(w_1), \ldots, \G(w_n)\}$,
    and its \emph{extent} as $\E(\vw) = n$.
\end{definition}

\subsection{Approximation relation}

A relation $\APPROXEQ$ has been defined
in~\cite{Beckmann:2002}.
Here we will consider the converse $\apprx$ of  $\APPROXEQ$.

\begin{definition}[Approximation relation]
  The \emph{approximation relation} $\apprx$ is a binary relation
  over $\Dom$, defined inductively as follows: 
  \begin{itemize}
  \item $*\apprx v$ for any $v \in \Dom$.
  \item $\epsilon \apprx \epsilon$.
  \item if $v_1 \apprx v_2$, then
    $v_10 \apprx v_20$ and $v_11 \apprx v_21$.
  \end{itemize}
  We pronounce `$v\apprx w$' as `$v$ approximates $w$'.
    
  We extend $\apprx$ to tuples:
  $(v_1,\dots,v_n)\apprx(w_1,\dots,w_n)$ iff $v_i\apprx w_i$ for each~$i$.
\end{definition}

\begin{proposition}
  $\apprx$ is a partial order on $\Dom^n$, that is, it is reflexive,
  anti-symmetric and transitive.
  \qed
\end{proposition}

\begin{definition}[Compatible]
  $u,v\in\Dom$ are \emph{compatible} if $u \apprx v$ or $v \apprx u$.
  $(u_1, \ldots, u_n)$ and $(v_1, \ldots, v_n)$ in $\Dom^n$ are
  compatible if each $u_i$ and $v_i$ are.
  We denote $\vu$ and $\vv$ being compatible with $\vu \compat \vv$.
\end{definition}

The following lemma has already been proven in~\cite{Beckmann:2002}:
\begin{lemma}\label{lem:compat}
  If $\vu,\vv,\vw \in \Dom^n$ and $\vu,\vv \apprx \vw$,
  then $\vu \compat \vv$.
  \qed
\end{lemma}

\begin{lemma}\label{lem:maxapprx}
  Any finite subset $S$ of $\Dom$ of pairwise compatible elements
  has a maximal element w.r.t.~$\apprx$ which we denote with
  $\maxapprx(S)$,
  where $\maxapprx(\emptyset)=*$.
  \qed
\end{lemma}

\begin{definition}[Generator]
  A \emph{generator} for $\Dom^n \to \Dom$ is a mapping $\vu \mapsto v$
  with $\vu \in \Dom^n$ and $v \in \Dom\setminus\{*\}$.
\end{definition}

\begin{definition}[Consistent set]\label{def:consis_set}
  A \emph{consistent set} $\Finite{f}$ of $\Dom^n \to \Dom$ is a
  finite set of generators satisfying the following condition:
  \begin{quote}
    if $\vx \mapsto y, \vu \mapsto v \in \Finite{f}$ and  $\vx \compat \vu$,
    then $y \compat v$.
  \end{quote}
\end{definition}


\begin{definition}[Finitely generated maps]
  A consistent set $\Finite{f}$ defines a mapping, which we call a \emph{finitely generated map} or just \emph{map},  via
  \begin{equation*}
    \Finite{f}(\vx) = \maxapprx \{ v \mid \exists \vw, \vw \apprx \vx
    \text{ and } \vw \mapsto v \in \Finite{f} \}. 
  \end{equation*}
  We sometimes write $\Finite{f}[\vx]$ to denote the set
  \[  
    \{ v \mid \exists \vw, \vw \apprx \vx
      \text{ and } \vw \mapsto v \in \Finite{f} \} 
  \]
  so that $\Finite{f}(\vx) = \maxapprx \Finite{f}[\vx]$.
\end{definition}

To see that maps are well-defined, consider two generators
$\vw \mapsto v$ and $\vw' \mapsto v'$ in $\Finite{f}$
with $\vw,\vw'\apprx \vx$.
With Lemma~\ref{lem:compat} we obtain $\vw\compat \vw'$.
Hence $v\compat v'$ as $\Finite{f}$ is consistent.
Thus, using Lemma~\ref{lem:maxapprx}, the set
\[
  \Finite{f}[\vx] = \{ v \mid \exists \vw, \vw \apprx \vx
    \text{ and } \vw \mapsto v \in \Finite{f} \}
\]
has a maximal element w.r.t.~$\apprx$.

\begin{lemma}[Expansion property of maps]\label{lem:expansion
    property of maps}
    Let $\Finite{f}_1$ and $\Finite{f}_2$ be consistent sets of
    $\Dom^n \to \Dom$.
    If $\Finite{f}_1 \subseteq \Finite{f}_2$, then
    $\Finite{f}_1(\vv) \apprx \Finite{f}_2(\vv)$
    for $\vv \in \Dom^n$.
\end{lemma}

\begin{proof}
  From $\Finite{f}_1 \subseteq \Finite{f}_2$ we immediately obtain
  $\Finite{f}_1[\vv] \subseteq \Finite{f}_2[\vv]$.
  Hence the assertion follows.
\end{proof}

For $\vx\apprx \vy$ we have $\Finite{f}[\vx]\subseteq\Finite{f}[\vy]$, thus we
obtain that finitely generated maps are monotone.


\begin{lemma}[Monotonicity of finitely generated maps]\label{lem:monontone maps}
  Finitely generated maps are monotone w.r.t.~$\apprx$. 
  \qed
\end{lemma}



\begin{remark}
  There are monotone maps which cannot be represented by finite consistent sets.  
  For example, the identity function from $\Dom$ to $\Dom$ is monotone but
  cannot be represented by a finite consistent set.
\end{remark}




\begin{definition}[Measures for consistent sets]
  We define two measures for consistent sets $\Finite{f}$:
  \begin{itemize}
  \item The gauge $\G(\Finite{f})$ is given as
    \[  \max \{ \G(\vv),\G(w) \mid \vv \mapsto w \in \Finite{f} \}
      \enspace. \]
  \item The extent $\E(\Finite{f})$ is given as
    \[ \max \{\E(\vv) \mid \vv \mapsto w \in \Finite{f} \}
      \enspace. \]
  \end{itemize}
\end{definition}

\begin{remark}
  Using the above measures, the length of $\Finite{f}$, $\lh(\Finite{f})$, in some natural serialization of $\Finite{f}$, can be bounded by
  $\lh(\Finite{f}) = \Oh(\# \Finite{f} \cdot \E(\Finite{{f}}) \cdot
    \G(\Finite{{f}}))$.
\end{remark}

\begin{definition}[Frame]
  A \emph{frame} $F$ is a partial, finite mapping of
  function symbols $f\in\cF\setminus\cB$
  to  consistent sets.
  We extend $F$ to all $f\in\cF\setminus\cB$
  by setting $F(f) = \bot$ for $f\notin\cB\cup\dom(F)$,
  where $\bot$ denotes the empty set $\emptyset$.
  
  The set of frames is partially ordered by 
  \begin{equation*}
    F_1 \forder F_2 \iff \forall f,  F_1(f) \subseteq F_2(f) \enspace.
  \end{equation*}

  A frame $F$ defines an \emph{evaluation} $F(f)(v)$ for
  $f\in\cF$ and $\vv\in\Dom^{\ar(f)}$ as follows:  
  \begin{itemize}
  \item If $f\in\cB$, let $F(f)(\vv) = f(\vv)$
  \item If $f\notin\cB$ and $F(f) = \Finite{f}$, let
    $F(f)(\vv) = \Finite{f}(\vv)$.
  \end{itemize}
  Observe that for $f \notin \dom(F)\cup\cB$,
  we have $F(f) = \bot$,
  hence $F(f)(\vv) = \bot(\vv) = *$.
\end{definition}

\begin{definition}[Measures for frames]
  We use the following measures for frames:
  \begin{itemize}
  \item The width of $F$ is given by
    $\W(F) = \max \{\# F(f) \mid f \in \dom(F) \}$.
  \item The gauge of $F$ is given by
    $\G(F) = \max \{ \G(f) \mid f \in \dom(F) \}$.
  \item The extent of $F$ is given by
    $\E(F) = \max \{ \E(f) \mid f \in \dom(F) \}$.
  \end{itemize}
\end{definition}

\begin{remark}
  Using the above measures, the length of $F$, $\lh(F)$, in some natural serialization of $F$, can be bounded by
  $\lh(F) = \Oh(\#\dom(F) \cdot \W(F) \cdot \E(F) \cdot \G(F))$ .
\end{remark}

\begin{definition}[Assignments]
  An \emph{assigment} $\rho$ is a partial, finite mapping from
  variables $\cX$ to approximations $\Dom$.
  We extend assignments outside their domain, by setting
  $\rho(x)=*$ for $x\notin\dom(\rho)$.

  We extend the approximation order $\apprx$ to assignments
  pointwise:
  \[
    \text{Let}\quad  \rho_1\apprx\rho_2 \quad\text{iff}\quad
    \forall x, \rho_1(x)\apprx\rho_2(x)\enspace.
  \]

  With $\rho[x\mapsto v]$ we denote the assignment that behaves like
  $\rho$ but maps variable~$x$ to approximation $v$:
  \[
    \rho[x\mapsto v](y)\quad=\quad
    \left\{
      \begin{array}{l@{\ \colon\ }l}
        v & \text{if } y=x \\
        \rho(y) & \text{otherwise.}
      \end{array}
    \right.
  \]

  We apply assignments to generalized variables in the natural way,
  e.g., $\rho(\suc_i(x)) = \suc_i(\rho(x))$.
\end{definition}

\begin{definition}[Measures for assignments]
  We use the following measures for assignments:
  \begin{itemize}
  \item
    The width of $\rho$ is given by
    $\W(\rho) = \# \dom(\rho)$.
  \item
    The gauge of $\rho$ is given by
    $\G(\rho) = \max\{ \G(v) \mid v\in\rng(\rho) \}$.
  \end{itemize}
\end{definition}

\begin{remark}
  Using the above measures, the length of $\rho$, $\lh(\rho)$, in some natural serialization of $\rho$, can be bounded by $\lh(\rho) \ =\ \Oh(\W(\rho) \cdot \G(\rho))$.
\end{remark}

\begin{definition}[Evaluation]
    Let $\rho$ be an assignment, $F$ a frame, and $t$ a term.
    The \emph{evaluation $\evalfr{t}$ of $t$ under $F,\rho$}
    is defined recursively as follows:
    \begin{align*}
      \evalfr{x} &\ =\ \rho(x) \quad\text{for a variable } x\in\cX;\\
      \evalfr{f(t_1, \dots, t_n)}
               &\ =\ F(f)(\evalfr{t_1}, \dots, \evalfr{t_n}) \enspace.
    \end{align*}
\end{definition}

We have the following immediate properties of evaluations.

\begin{lemma}
  \begin{enumerate}
  \item
    $\evalfr{t}\in\Dom$
  \item 
    $\evalfr{t}$ is monotone in $F$ and $\rho$ w.r.t.~$\forder$.
  \end{enumerate}
\end{lemma}

\begin{proof}
  (1) follows immediately from the definition.
  
  We prove (2) by induction on $t$, showing that
  for $F \sqsubseteq F'$ and $\rho \sqsubseteq \rho'$,
  \begin{equation*}
    \evalfr{t} \sqsubseteq \eval{t}_{F', \rho'} \enspace.
  \end{equation*}
  If $t \equiv x$, the assertion holds because
  $\rho(x) \sqsubseteq \rho'(x)$.
  If $t \equiv f(t_1, \dots, t_n)$, we compute
  \begin{multline*}
    \evalfr{f(t_1, \ldots, t_n)}
    \ =\ F(f)(\evalfr{t_1}, \dots, \evalfr{t_n})\\
    \ \sqsubseteq\ F'(f)(\evalfr{t_1}, \dots, \evalfr{t_n})\\
    \ \sqsubseteq\ F'(f)(\eval{t_1}_{F', \rho'}, \dots,
    \eval{t_n}_{F', \rho'})
    \ =\ \eval{t}_{F', \rho'} \enspace,
  \end{multline*}
  where the first approximation uses the expansion property of maps,
  Lemma~\ref{lem:expansion property of maps}, and the second the
  induction hypothesis and monotonicity of maps,
  Lemma~\ref{lem:monontone maps}. 
\end{proof}

\begin{lemma}\label{lem:valuebound}
  Let $\rho$ be an assignment, $F$ a frame, and $t$ a term.
  Then
  \begin{equation*}
    \G(\evalfr{t})\ \leq\ \max(\G(\rho), \G(F)) + \lh(t)
    \enspace.
  \end{equation*}
\end{lemma}

\begin{proof}
    By induction on $t$.
    If $t$ is a variable $x$, we have
    \[ \G(\evalfr{t})\ =\ \G(\rho(x))\ \le\ \G(\rho) \enspace. \]
    If $t$ is $\epsilon$ we compute
    $\G(\evalfr{t}) = 1 = \lh(t)$.
    For $t$ of the form $\suc_i(t_1)$ we have
    \begin{multline*}
      \G(\evalfr{t})\  =\ \G(\suc_i(\evalfr{t_1}))\ =\ \G(\evalfr{t_1})+1 \\
      \le\ \max(\G(\rho), \G(F)) + \lh(t_1) + 1
      \ =\ \max(\G(\rho), \G(F)) + \lh(t) \enspace .
    \end{multline*}

    Finally, assume $t$ is of the form $f(t_1, \ldots, t_n)$ with
    $f\notin\Basis$.
    Then we have $\G(\evalfr{t}) \leq \G(F)$.
\end{proof}

\begin{lemma}[Substitution Lemma]\label{lem:substitution}
  $\evalfr{t(u)} = \eval{t(x)}_{F, \rho[x \mapsto \evalfr{u}]}$
\end{lemma}

\begin{proof}
    The proof is by induction on $t$.
\end{proof}

\section{Frame Models}\label{sec:FrameModels}

In this section we develop frames into models for equational theories based on nice axioms.


\begin{definition}[Model]
  A frame $F$ is a \emph{model of $\Ax$} iff
  for any $t = u$ in $\Ax$ and any assignment $\rho$,
  $\evalfr{t} \apprx \evalfr{u}$.
\end{definition}

\begin{remark}
  In general, the notion of being a model cannot be expressed as a bounded formula, thus will
  usually be in $\Pi_1$, but not in \pib1.
\end{remark}

We restrict the notion of being a model to obtain a bounded property.
Let $\kappa$ be a positive integer, which is intended to bound the gauge
of approximations occurring in frames and assignments
that need to be considered in the definition of models.
Furthermore, we restrict the definition to axioms to those occurring in a
particular derivation $\cD$.

\textbf{For the remainder of this section,
we assume that $\kappa$ and $\cD$ are fixed.}
With $\var(\cD)$ we denote the variables occurring in~$\cD$.

\begin{definition}[$\kappa$-Model]
  A frame $F$ is a \emph{$\kappa$-model of $\cD$} iff
  $\G(F)\le\kappa$, and
  for any $t = u$ in $\Ax$ occurring in $\cD$ and any assignment
  $\rho$ with $\dom(\rho)\subseteq\var(\cD)$ and $\G(\rho)\le\kappa$,
  we have $\evalfr{t} \apprx \evalfr{u}$.
\end{definition}

\begin{remark}
  The notion of $F$ being a $\kappa$-model of $\cD$ can be written as a \pib1 formula.
\end{remark}

\begin{lemma}\label{lem:existence model}
    The empty frame is a $\kappa$-model of $\cD$.
\end{lemma}

\begin{proof}
  Let $F$ be the empty frame. 
  Consider $t=u$ in \Ax, and assignment~$\rho$.
  Then $t$ is of the form $f(\vt)$ for some $f$ in $\cF\setminus\cB$.
  We have $F(f)=\bot$, hence
  $\evalfr{f(\vt)} = * \apprx \evalfr{u}$.
\end{proof}

We will now define the notion of \emph{updates} that can be used to expand models based on axioms occurring in $\D$.

\begin{definition}[Updates]\label{def:updates}
  Let $F$ be a $\kappa$-model of $\cD$.
  An \emph{update based on $F$, $\kappa$ and $\cD$}
  is any $f\in\cF\setminus\cB$ and generator $\vv\mapsto w$,
  which we denote as $\updfvw$,
  such that $\G(\vv,w)\le\kappa$ and
  there exists $t=u$ in \Ax occurring in $\cD$
  and an assignment $\rho$
  satisfying that
  \begin{itemize}
  \item $t$ is of the form $f(\vt)$,
  \item $v_i=\rho(t_i)$  for $i\le\ar(f)$,
  \item and $w=\evalfr{u}$.
  \end{itemize}
  
  With $\FupdFfvw$ we denote the frame $F'$ given by
  \begin{align*}
    F'(g) &= F(g) \quad\text{if } g\neq f \\
    F'(f) &= F(f) \cup \{\vv\mapsto w\}
  \end{align*}
  
  The gauge of $\updfvw$, denoted $\G(\updfvw)$, is given by $\G(\vv,w)$, its extent, denoted $\E(\updfvw)$, by $\E(\vv)$.
\end{definition}


\begin{remark}
  The length of update $\updfvw$, $\lh(\updfvw)$, can be bounded by 
  \[ 
    \lh(\updfvw) \ =\ \Oh(\E(\updfvw) \cdot \G(\updfvw)) \enspace.
  \]
\end{remark}

\begin{remark}
  The arguments $\vt$ to $f$ above are generalized variables, as \Ax
  is nice.  
  Thus $\rho(t_i)$ is well-defined.
  Furthermore, in each term of $\vt$, at most one variable can occur,
  and such variables are distinct for different terms
  as \Ax is nice, as remarked before.
  Hence, an update uniquely determines an axiom in $\Ax$ and substitution 
  on which it is based.
\end{remark}


\begin{remark}
  For $F'=\FupdFfvw$ we compute 
  \begin{itemize}
    \item $\#(F') \le \#(F)+1$,
    \item $\W(F') \le \W(F)+1$,
    \item $\G(F') = \max\{\G(F),\G(\vv,w)\}$,
    \item $\E(F') = \max\{\E(F),\ar(f)\}$.
  \end{itemize}
\end{remark}

We now formulate and prove a crucial property of updates: 
They can be used to extend $\kappa$-models for $\D$.

\begin{proposition}[\Sot]\label{prop:model update}
  Let $F$ be a $\kappa$-model of $\cD$,
  $\updfvw$ an update based on  $F$, $\kappa$ and $\cD$, and
  $F' = \FupdFfvw$.
  Then $F'$ is a $\kappa$-model of $\cD$.
\end{proposition}

\begin{proof}
  We argue in \Sot.
  Let the assumption of the proposition be given, and assume that
  $\updfvw$ is given via $t=u$ in $\Ax$ and assignment $\rho$, where
  $t$ is of the form $f(\vt)$ and $\vv=\rho(\vt)$.
  W.l.o.g., $\dom(\rho)=\var(t)$.
  We have $\G(\rho)\le\kappa$.
  
  In order to show that $F'=\FupdFfvw$ is a $\kappa$-model for $\cD$,
  it suffices to show that
  \begin{enumerate}
  \item\label{Fupd1}
    $F'(f)$ is a consistent set, and
  \item\label{Fupd2}
    for any $t'=u'$ in \Ax occurring in $\cD$, and any assignment
    $\rho'$ with $\G(\rho')\le\kappa$, 
    we have $\evalfrp{t'}\apprx\evalfrp{u'}$.
  \end{enumerate}

  For \eqref{Fupd1}, consider $\vv'\mapsto w' \in F(f)$ such that
  $\vv' \compat \vv$.
  Then there exists $\vy$ such that $\vv,\vv'\apprx\vy$ and
  $\G(\vy)\le\kappa$ -- we can choose $y_i$ to be $\maxapprx\{v_i, v'_i\}$, hence $\G(y_i)\le\kappa$ follows from assumption $\G(\vv_i), \G(\vv'_i) \leq \kappa$.
  Choose $\hat\rho$ with $\dom(\hat\rho)=\var(t)$
  such that $y_i=\hat\rho(t_i)$,
  which is possible since $\vv\apprx\vy$.
  We observe that $\rho\apprx\hat\rho$ and that $\G(\hat\rho)\le\kappa$.

  Let $S$ be $F(f)[\vy]$, that is
  \[
    S \quad=\quad \{ \tilde{w} \mid \exists \tilde{v}, \tilde{v} \apprx \vy
       \text{ and } \tilde{v} \mapsto \tilde{w} \in F(f) \} \enspace.
  \]
  We have $w'\in S$ as $\vv'\apprx \vy$, hence
  \[
    w' \ \apprx\ \maxapprx S
    \ =\ F(f)(\vy) = \eval{t}_{F,\hat\rho} \\
    \ \apprx\ \eval{u}_{F,\hat\rho}
  \]
  as $F$ is a $\kappa$-model of $\cD$.
  Furthermore,
  \[ w \ =\  \evalfr{u}
    \ \apprx\ \eval{u}_{F,\hat\rho}
  \]
  as $\rho\apprx\hat\rho$.
  Hence $w\compat w'$ using Lemma~\ref{lem:compat}.

  For \eqref{Fupd2}, let  $t'=u'$ be in \Ax occurring in $\D$,
  and $\rho'$ be an assignment with $\G(\rho')\le\kappa$.
  If $t'=u'$ is not identical to $t=u$, then the assertion follows
  from $F$ being a $\kappa$-model of $\cD$:
  \Ax being nice implies $\evalfrp{t'}=\eval{t'}_{F,\rho'}$ in this case, hence
  \[
    \evalfrp{t'} \ =\ \eval{t'}_{F,\rho'}
    \ \apprx\ \eval{u'}_{F,\rho'} 
    \ \apprx\ \evalfrp{u'}
  \]
  as $F$ is a $\kappa$-model of $\cD$, and $F\forder F'$.

  Now consider $t'=u'$ being identical to $t=u$.
  Let $y_i$ be $\rho'(t_i)$.
  If $\vv\not\apprx\vy$, then again
  $\evalfrp{t'}=\eval{t'}_{F,\rho'}$ and the assertion follows from
  $F$ being a $\kappa$-model of $\cD$ as before.

  So assume $\vv\apprx\vy$.
  Let $\vx$ be the list of variables occuring in $t$, then we have
  $\rho{\restriction}_{\vx} \apprx \rho'{\restriction}_{\vx}$.
  We compute
  \[
    F'(f)(\vy) \ =\ \maxapprx F'(f)[\vy]
    \ =\  \maxapprx (\{w\}\cup F(f)[\vy])
    \ =\  \maxapprx \{ w, F(f)(\vy)\} \enspace.
  \]
  We consider $w$ and $F(f)(\vy)$ in turns:
  For $F(f)(\vy)$ we  have
  \[
    F(f)(\vy) \ =\ \eval{t}_{F,\rho'}
    \ \apprx\  \eval{u}_{F,\rho'}
  \]
  as $F$ is a $\kappa$-model of $\cD$.
  In case of $w$ we have,
  \[
    w \ =\  \evalfr{u}
    \ =\  \eval{u}_{F,\rho\restriction_{\vx}}
    \ \apprx\  \eval{u}_{F,\rho'\restriction_{\vx}}
    \ =\  \eval{u}_{F,\rho'} 
  \]
  using
  $\rho{\restriction}_{\vx} \apprx \rho'{\restriction}_{\vx}$.
  Hence $F'(f)(\vy) \apprx \eval{u}_{F,\rho'}$.
  Thus
  \[ \evalfrp{t} \ =\ F'(f)(\vy)
    \ \apprx\  \eval{u}_{F,\rho'}
    \ \apprx\  \evalfrp{u} \]
  as $F\forder F'$.
\end{proof}

\newcommand{\updfvwsub}[1]{\update{f_{#1}}{\vv_{#1}}{w_{#1}}}
\begin{definition}
  A \emph{sequence of updates based on $F$, $\kappa$ and $\cD$} is a
  sequence $\sigma$ of the form
  \[ \Seq{\, \updfvwsub1 \,,\dots,\, \updfvwsub{\ell} \,}
  \]
  such that for
  \begin{align*}
    F_0 &\ :=\ F\\
    F_{i+1} &\ :=\ F_i \updcons \updfvwsub{i+1}
  \end{align*}
  we have that
  \[
    \updfvwsub{i+1} \quad \text{is an update based on $F_i$,
      $\kappa$ and $\cD$.}
  \]
  
  Let ${F} \updcons {\sigma}$ denote $F_\ell$.
  The \emph{sequence length of $\sigma$}, denoted $\seqlh(\sigma)$, is given by~$\ell$.
  The gauge of $\si$ is given by $\G(\si)=\max\{\G(\vv_1,w_1),\dots,\G(\vv_\ell,w_\ell)\}$,
  its extend by $\E(\si)=\max\{\ar(f_1),\dots,\ar(f_\ell)\}$.
\end{definition}

\begin{remark}
  The length of $\si$, $\lh(\si)$, can be bounded by $\lh(\si)=\Oh(\E(\si)\cdot\G(\si)\cdot\seqlh(\si))$.
\end{remark}

\begin{remark}
  For $F'=F\updcons\sigma$ we compute 
  \begin{itemize}
    \item $\#(F') \le \#(F)+\seqlh(\si)$,
    \item $\W(F') \le \W(F)+\seqlh(\si)$,
    \item $\G(F') = \max\{\G(F),\G(\si)\}$,
    \item $\E(F') = \max\{\E(F),\E(\si)\}$.
  \end{itemize}
\end{remark}

\begin{corollary}[\Sot]\label{cor:models}
  Assuming the notions given by the previous definition, 
  all $F_i$'s are $\kappa$-models of $\cD$, for $i\le\ell$.
\end{corollary}

\begin{proof}
  The proof is by induction on $i\le\ell$ using Proposition~\ref{prop:model update}.
\end{proof}

\section{Soundness in \texorpdfstring{\Stt}{S22}}\label{sec:SoundnessInS22}

We now prove a soundness property for equational reasoning using approximation semantics.  
The proof will be formalizable in \Stt. 
This will be improved in the remaining sections to a proof formalizable in \Sot by introducing an additional property.
To keep the exposition clearer, we first prove soundness based on the notions introduced so far.

\begin{theorem}[\Stt]\label{thm:soundness in S22}
  Assume $\D \vdash t = u$ is in Variable Normal Form.
  Let $\rho$ be an assignment,
  and $F$ a model for~\Ax.
  Let $\kappa$ be $\max\{\G(F),\G(\rho)\}+\lh(\cD)$.
  Then there are sequences $\sigma_1$ and $\sigma_2$ of updates
  based on $F$, $\kappa$ and $\cD$ 
  such that
  \begin{align*}
    \evalfr{t} &\ \apprx\ \eval{u}_{F\updcons \sigma_1, \rho} \\
    \evalfr{u} &\ \apprx\ \eval{t}_{F\updcons \sigma_2, \rho}
  \end{align*}
\end{theorem}

To prove the previous theorem, we consider the following more
general claim.

\begin{claim}[\Stt]\label{claim:soundness in S22}
  Fix some derivation $\cD$ in Variable Normal Form, 
  some model $F$ for~\Ax,
  and some integer $U$ such that $\G(F)+\lh(\cD)\le U$.

  Let $\D_0 \vdash t = u$ be a sub-derivation of $\D$.
  Let $\rho$ be an assignment,
  and $\si$ a sequence of updates based on $F$, $U$, $\cD$,
  satisfying
  \begin{align*}
    \dom(\rho)\ &\subseteq\ \var(\cD) \\
    \G(\rho), \G(\si), \E(\si), \seqlh(\si)\ &\le\ U - \lh(\cD_0)
  \end{align*}
  
  Then there are sequences $\si_1$ and $\si_2$ of updates based on
  $F$, $U$, $\cD$ with
  \begin{align*}
    \E(\si_i), \seqlh(\si_i)\ &\leq\ \lh(\D_0)\\
    \G(\si_i)\ &\leq\ \max\{\G(F),\G(\si),\G(\rho)\} + \lh(\cD_0)
  \end{align*}
  such that
  \begin{align*}
    \evalfr{t} &\ \apprx\ \eval{u}_{F' \updcons \sigma_1, \rho} \\
    \evalfr{u} &\ \apprx\ \eval{t}_{F' \updcons \sigma_2, \rho}
  \end{align*}
  for $F' = F\updcons \si$.
\end{claim}

The Theorem follows from the Claim by letting $\cD_0=\cD$,
$\rho$ as given,
$\si=\Seq{}$, and $U=\max\{\G(F),\G(\rho)\}+\lh(\cD)$.

\begin{proof}[Proof of Claim~\ref{claim:soundness in S22}]
  We argue in \Stt.
  Let $\cD$, $F$ and $U$
  be given as in the Claim.
  We prove that for any $\cD_0$, $\rho$, $\si$ satisfying the
  conditions of the Claim, there are $\si_1$ and $\si_2$ satisfying the
  assertion of the Claim, by induction on $\lh(\cD_0)$.
  Thus this is proven by logarithmic induction (LIND) on a $\pib2$-property, which is
  available in \Stt by Theorem~\ref{thm:Pi-LIND}.

  Let $\cD_0$, $\rho$, $\si$ be given, that are 
  satisfying the conditions in the Claim.
  Let $F'$ be $F\updcons\si$. then
  $F'$ is a $U$-model of $\cD$ by Corollary~\ref{cor:models}.

  We now consider cases according to the last rule applied
  in~$\cD_0$.
  If that is the \textbf{Reflexivity Rule} $\vdash t=t$, we can choose
  $\si_1=\si_2=\Seq{}$ to satisfy the assertion of the Claim.

  \bigskip

  \paragraph{\bf Axiom Rule:}
  More interesting is the case of Axiom Rule $\vdash t=u$.
  As $\D$ is in Variable Normal Form we have that $t=u$ is an injective renaming of an equation in \Ax.
  W.l.o.g.\ we can assume that $t=u$ is in \Ax, as renamings of
  variables would make no difference to the following argument.
  As \Ax is nice, we have that $t$ is of the form $f(\vt)$ for some
  $f\in\cF\setminus\cB$ and generalized variables $\vt$.
  Let $v_i$ be $\rho(t_i)$ and $w=\eval{u}_{F',\rho}$.
  We compute $\G(v_i)\le\G(\rho)+1$
  and, using Lemma~\ref{lem:valuebound},
  \[
    \G(w)
    \ \le\ \max\{\G(\rho),\G(F')\} + \lh(u)
    \ <\ \max\{\G(\rho),\G(F),\G(\si)\} + \lh(\cD_0) \enspace.
  \]
  
  Let $\si_1=\Seq{}$ and $\si_2=\Seq{\updfvw}$, then
  \begin{align*}
  \E(\si_i)\ &\le\ \ar(f)\ <\ \lh(\cD_0) \enspace,\\
  \seqlh(\si_i)\ &\le\ 1\ <\ \lh(\cD_0) \enspace,\text{ and}\\
  \G(\si_i)\ &\le\ \max\{\G(\si),\G(\rho),\G(F)\} + \lh(\cD_0)\enspace.
  \end{align*}
  Furthermore, $\eval{t}_{F',\rho}\apprx\eval{u}_{F',\rho}$
  as $F'$ is $U$-model of $\cD$, which proves the assertion for
  $\si_1$.
  For $\si_2$, let $F''$ be $F'\updcons\si_2$, then we have
  \begin{align*}
    \eval{u}_{F',\rho}\ =\ w\ \apprx\ \maxapprx F''(f)[\vv]\ &=\ F''(f)(\vv) \\
    &=\  F''(f)(\dots,\rho(t_i),\dots)
    \ =\ \eval{t}_{F'',\rho}  \enspace.
  \end{align*}

  \bigskip
  
  \paragraph{\bf Symmetry Rule:}
  For the case of Symmetry Rule, let $\cD_1$ be the sub-derivation of
  $\cD_0$ ending in $u=t$.
  By induction hypothesis we obtain $\si'_1$ and $\si'_2$ satisfying
  the assertion for $\cD_1$.
  By choosing $\si_1=\si'_2$ and $\si_2=\si'_1$ we immediately fulfill
  the assertion for $\cD_0$.

  \bigskip
  
  \paragraph{\bf Transitivity Rule:}
  If $\cD_0$ ends with an application of the Transitivity Rule, it
  must be of the form
    \begin{prooftree}
    \AxiomC{$\cD_1$}
    \noLine
    \UnaryInfC{$t=s$}
    \AxiomC{$\cD_2$}
    \noLine
    \UnaryInfC{$s=u$}
    \BinaryInfC{$t=u$}
  \end{prooftree}
  By induction hypothesis applied to $\cD_1$, $\rho$ and $\si$,
  we obtain some $\si^1_1$ satisfying
  \begin{align*}
    \E(\si^1_1), \seqlh(\si^1_1)\ &\le\ \lh(\cD_1) \enspace, \\
    \G(\si^1_1)\ &\le\ \max\{\G(\si),\G(\rho),\G(F)\} + \lh(\cD_1)
                   \enspace, \text{ and } \\
    \evalfr{t}\ &\apprx\ \evalr{s}{F^1_1}
                  \ \text{ for } F^1_1=F'\updcons \si^1_1\enspace.
  \end{align*}
  We compute
  \begin{align*}
    \E(\si\updcons \si^1_1)
    \ &\le\ \max\{\E(\si),\E(\si^1_1)\}
    \ \le\ U-\lh(\cD_0) + \lh(\cD_1)
    \ <\  U - \lh(\cD_2) \\
    \seqlh(\si\updcons \si^1_1)
    \ &\le\ \seqlh(\si)+\seqlh(\si^1_1)
    \ \le\ U-\lh(\cD_0) + \lh(\cD_1)
    \ <\  U - \lh(\cD_2)
  \end{align*}
  and
  \begin{multline*}
    \G(\si\updcons \si^1_1)
    \ \le\ \max\{ \G(F),\G(\rho),\G(\si)\}+\lh(\cD_1) \\
    \ \le\ U-\lh(\cD_0) + \lh(\cD_1)
    \ <\  U - \lh(\cD_2)
  \end{multline*}
  because $\lh(\cD_0)>\lh(\cD_1)+\lh(\cD_2)$.
  Thus, we can apply i.h.\ to $\cD_2$, $\rho$ and
  $\si\updcons\si^1_1$, obtaining $\si^2_1$ satisfying
  \begin{align*}
    \E(\si^2_1),\seqlh(\si^2_1)\ &\le\ \lh(\cD_2) \enspace, \\
    \G(\si^2_1)\ &\le\ \max\{\G(\rho),\G(F), \G(\si\updcons\si^1_1)\} + \lh(\cD_2)
                   \enspace, \text{ and } \\
    \evalr{s}{F^1_1}\ &\apprx\ \evalr{u}{F^1_1\updcons \si^2_1}\enspace.
  \end{align*}
  Let $\si_1$ be $\si^1_1\seqcons\si^2_1$, then 
  we compute
  \begin{align*}
    \E(\si_1)\ &=\ \max\{\E(\si^1_1),\E(\si^2_1)\}
    \ \le\ \lh(\cD_1) + \lh(\cD_2)
    \ <\  \lh(\cD_0)\\
    \seqlh(\si_1)\ &=\ \seqlh(\si^1_1)+\seqlh(\si^2_1)
    \ \le\ \lh(\cD_1) + \lh(\cD_2)
    \ <\  \lh(\cD_0)
  \end{align*}
  and
  \begin{align*}
    \G(\si_1)\ &=\ \max\{ \G(\si^1_1), \G(\si^2_1) \} \\
    \ &\le\ \max\{ \G(\si^1_1),
         \max\{\G(\rho),\G(F), \G(\si\seqcons\si^1_1)\} + \lh(\cD_2) \} \\
    \ &=\ \max\{\G(\rho),\G(F), \G(\si),\G(\si^1_1)\} + \lh(\cD_2)\qquad \\
    \ &\le\ \max\{ \G(F),\G(\rho),\G(\si)\}+\lh(\cD_1) + \lh(\cD_2) \\
    \ &<\ \max\{ \G(F),\G(\rho),\G(\si)\}+\lh(\cD_0)
  \end{align*}
  Furthermore, $\evalfr{t}\ \apprx\ \evalr{s}{F^1_1}\ \apprx\
  \evalr{u}{F^1_1\updcons \si^2_1}\ =\  \evalr{u}{F'\updcons \si_1}$,
  because
  \[
    F^1_1\updcons \si^2_1
    \ =\ (F'\updcons \si^1_1) \updcons \si^2_1
    \ =\ F'\updcons (\si^1_1 \seqcons \si^2_1)
    \ =\ F'\updcons \si_1
  \]
  which proves the assertion for $\si_1$.
  The construction for $\si_2$ is similar, starting with~$\cD_2$.

  \bigskip

  \paragraph{\bf Compatibility Rule}
  In case of the last rule being the Compatibility Rule, $\cD_0$ will
  have the following form:
  \begin{prooftree}
    \AxiomC{$\cD_1$}
    \noLine
    \UnaryInfC{$t=u$}
    \UnaryInfC{$s[t/x]=s[u/x]$}
  \end{prooftree}
  Applying the i.h.\ to $\cD_1$, $\rho$ and $\si$, we obtain $\si_1$
  and $\si_2$ satisfying the assertion for~$\cD_1$.
  Let $\rho^1_1 = \rho[x\mapsto\evalfr{t}]$ and
  $\rho^2_1 = \rho[x\mapsto\evalr{u}{F'\updcons\si_1}]$.
  Then we have $\rho^1_1\apprx\rho^2_1$.
  Employing the Substitution Lemma~\ref{lem:substitution}, we obtain
  \[
    \evalfr{s[t/x]}
    \ = \ \eval{s}_{F,\rho^1_1}
    \ \apprx\ \eval{s}_{F'\updcons\si_1,\rho^2_1}
    \ =\ \eval{s[u/x]}_{F'\updcons\si_1,\rho} 
  \]
  which shows that $\si_1$ also satisfies the assertion for $\cD_0$.
  Similar for $\si_2$ and $\cD_0$.

  \bigskip

  \paragraph{\bf Substitution Rule}
  If $\cD_0$ ends in an application of Substitution, it will have the
  following form:
  \begin{prooftree}
    \AxiomC{$\cD_1$}
    \noLine
    \UnaryInfC{$t=u$}
    \UnaryInfC{$t[s/x]=u[s/x]$}
  \end{prooftree}
  We only consider the case that $x$ is occurring in $t=u$, the other
  case is trivial.
  
  Let $\rho'$ be $\rho[x\mapsto\evalfr{s}]$, then clearly
  $\dom(\rho')\subseteq\var(\cD)$.
  Furthermore, using Lemma~\ref{lem:valuebound}, we obtain
  \begin{align*}
    \G(\rho')\ &\le\ \max\{\G(\rho),\G(\evalfr{s})\} \\
    \ &\le\ \max\{ \G(\rho), \max\{\G(\rho),\G(F)\}+\lh(s) \} \\
    \ &=\ \max\{\G(\rho),\G(F)\}+\lh(s)
  \end{align*}
  hence
  \[
    \G(\rho')\ \le\ U-\lh(\cD_0) +\lh(s)
    \ <\ U-\lh(\cD_1)
  \]
  because
  \[
    \lh(\cD_0)\ =\ \lh(\cD_1) + \lh(t[s/x]=u[s/x])
    \ >\ \lh(\cD_1) + \lh(s) \enspace.
  \]
  Thus we can apply the i.h.\ to $\cD_1$, $\rho'$ and $\si$, obtaining
  $\si_1$ and $\si_2$ such that
  \[
    \E(\si_i),\seqlh(\si_i)\ \le\ \lh(\cD_1) \ <\ \lh(\cD_0)
  \]
  and
  \begin{align*}
    \G(\si_i)\ &\le\
                 \max\{\G(F),\G(\si),\G(\rho')\} + \lh(\cD_1)  \\
               &\le\ \max\{\G(F),\G(\si),\G(\rho)\} + \lh(s) + \lh(\cD_1)  \\
    &<\ \max\{\G(F),\G(\si),\G(\rho)\} + \lh(\cD_0) 
  \end{align*}
  Furthermore,
  \[
    \eval{t}_{F,\rho'}\ \apprx\ \eval{u}_{F'\updcons\si_1,\rho'}
  \]
  Now we can compute, employing the Substitution
  Lemma~\ref{lem:substitution},
  \begin{align*}
    \evalfr{t[s/x]}
    \ &= \ \eval{t}_{F,\rho'}
        \ \apprx\ \eval{u}_{F'\updcons\si_1,\rho'} \\
        \ &=\ \eval{u}_{F'\updcons\si_1,\rho[x\mapsto\evalfr{s}]} \\
    \ &\apprx\ \eval{u}_{F'\updcons\si_1,\rho[x\mapsto\eval{s}_{F'\updcons\si_1,\rho}]}    
    \ =\ \eval{u[s/x]}_{F'\updcons\si_1,\rho} 
  \end{align*}
  which proves the assertion for $\si_1$ and $\cD_0$.
  Similar for $\si_2$ and $\cD_0$.
\end{proof}


\begin{corollary}\label{cor:Stt}
  The consistency of $\PETS(\Ax)$ is provable in $\Stt$.
\end{corollary}

\begin{proof}
  We argue in \Stt.
  Assume $\D$ is a $\PETS(\Ax)$ derivation ending in $0=1$.
  Using Proposition~\ref{prop:VNF} we can assume that $\D$ is in
  Variable Normal Form.
  Let $\rho$ be the empty assignment, and $F$ the empty model
  for~$\Ax$.
  Let $\kappa=\lh(\D)$.
  By the previous Theorem~\ref{thm:soundness in S22}, there is a
  sequence $\si_1$ of updates based on $F$, $\kappa$ and $\D$
  such that
  \[
    0 \quad=\quad \evalfr{0}
    \quad\apprx\quad  \eval{1}_{F\updcons \sigma_1, \rho}
    \quad=\quad 1
  \]
  which is impossible.
\end{proof}

\section{Instructions and Frame Models}\label{sec:instructions}

In order to be able to prove our main theorem in \Sot, we need to turn a
proof tree consisting of equations using rules for equational
reasoning into some linear sequence which describes how
terms are transformed step by step in order to go from the term
on the left-hand side of the final equation in the proof tree to the
term on the right-hand side, and visa versa.
This idea is similar to the one used in~\cite{Beckmann:2002} where such
proof trees (without the Substitution Rule) were turned into paths of a
corresponding term rewriting relation.
Here we
use instructions storing the operation that should be applied to a term
while moving through the tree.

We start by naming the instructions that will be considered.

\begin{definition}[Instructions]\label{def:instructions}
  We define a set of \emph{instructions} and their \emph{length} as
  follows: 
  \begin{description}
  \item[Axiom]
    $\IAx[t \rightarrow u]$ and $\IAx[t \leftarrow u]$
    are instructions,
    for any axiom $t=u\in\Ax$.
    Their length is $\lh(t)+\lh(u)+1$.
  \item[Substitution] $\ISup[s,t/x]$ and $\ISdwn[s,t/x]$
    are instructions,
    for terms $s,t$ and variable~$x$.
    Their length is $\lh(s)+\lh(t)+1$.
  \end{description}
  Sequences of instructions will be denoted with~$\tau$.
  With $\seqlh(\tau)$ we denote the sequence length of $\tau$, that is the number of instructions occurring in~$\tau$.
  With $\lh(\tau)$ we denote the length of $\tau$ given as the sum of
  lengths of instructions occurring in them.
\end{definition}

We now define the process of turning derivations into sequences of related instructions.

\begin{definition}\label{def:sequences of instructions}
  For a derivation $\D$, we define \emph{sequences of instructions}
  $\InstR_\D$ and $\InstL_\D$ by recursion on $\D$ as follows.
  \begin{description}
  \item[Axiom Rule]
    If $D$ is of the form
    \begin{prooftree}
        \AxiomC{}
        \UnaryInfC{$t = u$}
    \end{prooftree}
    let
    \begin{align*}
        \InstR_\D &\ :=\  \Seq{\IAx[t \rightarrow u]} \\
        \InstL_\D &\ :=\  \Seq{\IAx[t \leftarrow u]}
    \end{align*}

  \item[Reflexivity Rule]
    If $D$ is of the form
    \begin{prooftree}
        \AxiomC{}
        \UnaryInfC{$t = t$}
    \end{prooftree}
    let
    \[
      \InstR_\D \ :=\ \InstL_\D \ :=\  \Seq{}
    \]

  \item[Symmetry Rule]
    Consider $\D$ of the form
    \begin{prooftree}
      \AxiomC{$\D_1$}
      \noLine
      \UnaryInfC{$u = t$}
      \UnaryInfC{$t = u$}
    \end{prooftree}
    Let $\InstR_{\D_1}$ and  $\InstL_{\D_1}$ be given by i.h.,
    then define
    \begin{align*}
      \InstR_\D &\ :=\  \InstL_{\D_1}\\
      \InstL_\D &\ :=\  \InstR_{\D_1}
    \end{align*}

  \item[Transitivity Rule]
    Consider $\D$ of the form
    \begin{prooftree}
      \AxiomC{$\D_1$}
      \noLine
      \UnaryInfC{$t=s$}
      \AxiomC{$\D_2$}
      \noLine
      \UnaryInfC{$s=u$}
      \BinaryInfC{$t=u$}
    \end{prooftree}
    Let $\InstR_{\D_1}$, $\InstL_{\D_1}$, $\InstR_{\D_2}$ and
    $\InstL_{\D_2}$ be given by i.h.
    Define
    \begin{align*}
      \InstR_\D &\ :=\  \InstR_{\D_1} \seqcons \InstR_{\D_2}\\
      \InstL_\D &\ :=\  \InstL_{\D_2} \seqcons \InstL_{\D_1}.
    \end{align*}

  \item[Compatibility Rule]
    If $\D$ is of the form
    \begin{prooftree}
      \AxiomC{$\cD_1$}
      \noLine
      \UnaryInfC{$t=u$}
      \UnaryInfC{$s[t/x]=s[u/x]$}
    \end{prooftree}
    Let $\InstR_{\D_1}$ and  $\InstL_{\D_1}$ be given by i.h.,
    then define
    \begin{align*}
      \InstR_\D &\ :=\  \InstR_{\D_1}\\
      \InstL_\D &\ :=\  \InstL_{\D_1}
    \end{align*}

  \item[Substitution Rule]
    If $\D$ is of the form
    \begin{prooftree}
      \AxiomC{$\cD_1$}
      \noLine
      \UnaryInfC{$t=u$}
      \UnaryInfC{$t[s/x]=u[s/x]$}
    \end{prooftree}
    then let
    \begin{align*}
      \InstR_\D &\ :=\  \ISup[t,s/x] \seqadd \InstR_{\D_1}
                  \seqadd \ISdwn[u,s/x]\\
      \InstL_\D  &\ :=\  \ISup[u,s/x] \seqadd \InstL_{\D_1}
                   \seqadd \ISdwn[t,s/x]
    \end{align*}
  \end{description}
\end{definition}

\begin{remark}
  We observe that $\lh(\InstR_\D)=\lh(\InstL_\D)\le\lh(\D)$.
\end{remark}

We will now describe a process of evaluating terms using
approximations along sequences of instruction.
We start with the most basic and also most interesting step of the
reverse direction of an axiom instruction.

\textbf{For the remainder of this section,
we assume that $\kappa$ and $\cD$ are fixed.}

\begin{definition}\label{def:Psi}
  Let $t = u$ be an axiom in $\cD$,
  $\rho$ an assignment,
  and $F$ a $\kappa$-model of~\Ax.
  Define $\Psi(t \leftarrow u, \Seq{F, \rho})$ to be
  $\updfvw$ satisfying
  \begin{itemize}
  \item 
    $t$ is of the form $f(\vt)$ for some terms $\vt$;
  \item
    $v_i=\rho(t_i)$ for $i\le\ar(f)$;
  \item
    and $w=\evalfr{u}$.
  \end{itemize}
  For a sequence $\si$ of updates based on $F$, $\kappa$ and $\D$ we let
  $\Psi(t \leftarrow u, \Seq{F, \si, \rho})$ be
  $\Psi(t \leftarrow u, \Seq{F\updcons\si, \rho})$.
\end{definition}

\begin{lemma}\label{lem:Psi}
  Let $t = u$ be an axiom in \Ax, $\kappa'$ a positive integer with $\kappa'\le\kappa-\lh(u)$, 
  $\rho$~an assignment with $\G(\rho)\le\kappa'$, 
  and $F$ a $\kappa$-model of $\D$ with $\G(F)\le\kappa'$.
  Let $\updfvw$ be given by $\Psi(t \leftarrow u, \Seq{F, \rho})$.
  Then $\updfvw$ is an update based on $F$, $\kappa$ and $\D$,
  satisfying that $\G(\vv,w)\le\kappa'+\lh(u)$ and
  \[
    \evalfr{u}\ \apprx\ \eval{t}_{\FupdFfvw,\rho} \enspace.
  \]
\end{lemma}
  
\begin{proof}
  As \Ax is nice, we have that $t$ is of the form $f(\vt)$ for some
  $f\in\cF\setminus\cB$ and generalized variables $\vt$ (see Definition~\ref{def:generalized var}).
  Then $v_i=\rho(t_i)$ and $w=\eval{u}_{F,\rho}$.
  We compute $\G(v_i)\le\G(\rho)+1\le\kappa'+\lh(u)\le\kappa$,
  and, using Lemma~\ref{lem:valuebound},
  \[
    \G(w)
    \ \le\ \max\{\G(\rho),\G(F)\} + \lh(u)
    \ \le\ \kappa' + \lh(u) \ \le\ \kappa
  \]
  Hence, $\updfvw$ is an update based on $F$, $\kappa$ and $\D$.

  Furthermore, for $F' = \FupdFfvw$, we have
  \begin{align*}
    \eval{u}_{F,\rho}\ =\ w\ \apprx\ \maxapprx F'(f)[\vv]
     \ &=\ F'(f)(\vv) \\
        &=\ F'(f)(\dots,\rho(t_i)\dots) 
          \ =\ \eval{t}_{F',\rho}
          \enspace .                                                    
  \end{align*}
\end{proof}

\begin{definition}
  Let $\tau$ be a sequence of instructions,
  $\rho$ an assignment,
  $F$ a $\kappa$-model for $\D$,
  and $\si$ a sequence of updates based on $F$, $\kappa$ and $\D$.
  Let $\al=\Seq{F,\si,\rho}$.
  We define $\Phi(\tau,\al) = \Seq{F,\si',\rho'}$ by
  induction on $\tau$:

  If $\tau$ is the empty sequence,
  let $\Phi(\Seq{}, \al) = \al$

  Otherwise, $\tau$ is of the form $\tau'\seqadd I$ for some
  instruction~$I$.
  Let $\Seq{F,\si',\rho'}=\Phi(\tau',\al)$ by i.h., and let $F'$ be
  $F\updcons\si'$.
  We consider cases according to the form of $I$:
  \begin{description}
  \item[Axiom]
    For $I=\IAx[t \rightarrow u]$ let
    $\Phi(\tau,\al) = \Seq{F,\si',\rho'}$.

    For $I=\IAx[t \leftarrow u]$ let
    $\nu = \Psi(t \leftarrow u, \Seq{F,\si',\rho'})$,
    and define
    \[
      \Phi(\tau,\al) \ =\ \Seq{F,\si'\updcons\nu,\rho'} \enspace.
    \]


  \item[Substitution]
    If $I=\ISup[t,s/x]$, let
    \[
      \Phi(\tau,\al) \ =\ 
      \Seq{F,\si',\rho'[x\mapsto\eval{s}_{F',\rho'}]}\enspace.
    \]
    
    If $I=\ISdwn[t,s/x]$, let
    $\rho''$ be $\rho'$ but with $x$ removed from its domain:
    $\rho'' = \rho'\restriction_{\dom(\rho')\setminus\{x\}}$.
    Then let
    \[
      \Phi(\tau,\al) \ =\ \Seq{F,\si',\rho''}  \enspace.
    \]
  \end{description}
\end{definition}

%

%

\begin{lemma}\label{lem:instruction update measures}
  Let $\tau$ be a sequence of instructions for $\D$,
  $\rho$ an assignment,
  $F$ a $\kappa$-model of $\D$,
  and $\si$ a sequence of updates based on $F,\kappa$ and $\D$,
  satisfying
  \[
    \max\{\G(\rho),\G(F),\G(\si)\} + \lh(\tau) \ \le\  \kappa  \enspace.
  \]
  Let $\Seq{F,\si',\rho'}$ be $\Phi(\tau, \Seq{F,\si,\rho})$.
  Then we have
  \begin{enumerate}
  \item
    $\si'$ is a sequence of updates based on $F$, $\kappa$ and $\D$;
  \item
    $\seqlh(\si') \ \le\ \seqlh(\si) + \seqlh(\tau)$;
  \item
    $\G(\rho'),\G(\si') \ \le\  \max\{\G(\rho),\G(F),\G(\si)\}+\lh(\tau)$.
  \item
    $\E(\si') \ \le\ \E(\si) + \lh(\tau)$;
  \end{enumerate}
\end{lemma}

\begin{proof}
  Let $\tau_i$ be the sequence consisting of the first $i$ elements in~$\tau$, for $i=0,\dots,\seqlh(\tau)$.
  Let $\Seq{F,\si_i,\rho_i}$ be $\Phi(\tau_i, \Seq{F,\si,\rho})$.
  We can show by induction on~$i$ that
  \begin{enumerate}
  \item\label{lem:prf:instruction update measures - 1}
    $\si_i$ is a sequence of updates based on $F$, $\kappa$ and $\D$;
  \item\label{lem:prf:instruction update measures - 2}
    $\seqlh(\si_i) \ \le\ \seqlh(\si) + \seqlh(\tau_i)$;
  \item\label{lem:prf:instruction update measures - 3}
    $\G(\rho_i),\G(\si_i) \ \le\  \max\{\G(\rho),\G(F),\G(\si)\}+\lh(\tau_i)$.
  \item\label{lem:prf:instruction update measures - 4}
    $\E(\si_i) \ \le\ \E(\si) + \lh(\tau_i)$;
  \end{enumerate}
  
  For $i=0$ there is nothing to show as $\si_0=\si$ and $\rho_0=\rho$.

  In the induction step from $i$ to $i+1$ we have $\tau_{i+1}=\tau_i\seqadd I$ 
  for some instruction~$I$.
  We consider cases according to $I$.
  
  If $I=\IAx[t \rightarrow u]$ or $I=\ISdwn[t,s/x]$, there is nothing to show 
  as $\si_{i+1}=\si_i$ and $\G(\rho_{i+1})\le\G(\rho_i)$.
  
  In case $I=\ISup[t,s/x]$ we have $\si_{i+1}=\si_i$ and 
  $\rho_{i+1}=\rho_i[x\mapsto\eval{s}_{F\updcons\si,\rho_i}]$.
  Thus assertion \eqref{lem:prf:instruction update measures - 1} and 
  \eqref{lem:prf:instruction update measures - 2} follow immediately from i.h.
  For assertion \eqref{lem:prf:instruction update measures - 3} we compute,
  using Lemma~\ref{lem:valuebound}, 
  \begin{align*}
    \G(\eval{s}_{F\updcons\si,\rho_i}) 
    \ &\le\ \max\{\G(F),\G(\si_i),\G(\rho_i)\}+\lh(s) \enspace.
  \end{align*}
  Hence, using i.h.
  \begin{align*}
    \G(\rho_{i+1}) 
      \ &\le\ \max\{\G(\rho_i),\G(\eval{s}_{F\updcons\si,\rho_i})\} \\
      \ &\le\ \max\{\G(F),\G(\si_i),\G(\rho_i)\}+\lh(s) \\
      \ &\le\ \max\{\G(F),\G(\si),\G(\rho)\}+\lh(\tau_i)+\lh(s) \\
      \ &<\ \max\{\G(F),\G(\si),\G(\rho)\}+\lh(\tau_{i+1}) \enspace.
  \end{align*}

  In case $I=\IAx[t \leftarrow u]$ we have $\rho_{i+1}=\rho_i$.
  Let
  $\nu = \Psi(t \leftarrow u, \Seq{F,\si_i,\rho_i})$.
  By Lemma~\ref{lem:Psi} we obtain that $\nu$ is an update
  based on $F$, $\kappa$ and $\D$, 
  and that
  \[
    \G(\nu) \ \le\ \max\{\G(F),\G(\si_i),\G(\rho_i)\}+\lh(u)  \enspace.
  \]
  The former immediately implies assertion \eqref{lem:prf:instruction update measures - 1} 
  for $\si_{i+1}$.
  The latter implies,
  using i.h.
  \begin{align*}
    \G(\si_{i+1}) 
      \ &\le\ \max\{\G(\si_i),\G(\nu)\} \\
      \ &\le\ \max\{\G(F),\G(\si_i),\G(\rho_i)\}+\lh(u) \\
      \ &\le\ \max\{\G(F),\G(\rho),\G(\si)\}+\lh(\tau_i)+\lh(u) \\
      \ &<\ \max\{\G(F),\G(\rho),\G(\si)\}+\lh(\tau_{i+1}) \enspace.
  \end{align*}
  Thus assertion \eqref{lem:prf:instruction update measures - 3} follows.

  For assertion \eqref{lem:prf:instruction update measures - 2} we compute using i.h.
  \begin{align*}
    \seqlh(\si_{i+1})\ &=\ \seqlh(\si_i)+1 \\
      \ &\le\ \seqlh(\si)+\seqlh(\tau_i)+1 
      \ =\ \seqlh(\si)+\seqlh(\tau_{i+1}) \enspace.
  \end{align*}

  For assertion \eqref{lem:prf:instruction update measures - 4} we compute using i.h.
  \begin{align*}
    \E(\si_{i+1})\ &=\ \max\{\E(\si_i),\E(\nu)\} 
      \ \le\ \E(\si)+\lh(\tau_i)+ \lh(t)
      \ <\ \E(\si)+\lh(\tau_{i+1}) \enspace.
  \end{align*}
\end{proof}

\begin{lemma}
  Consider $\tau=\tau_1\seqcons\tau_2$.
  Then
  \[
    \Phi(\tau,\Seq{F,\si,\rho}) \quad=\quad
    \Phi(\tau_2, \Phi(\tau_1,\Seq{F,\si,\rho}))
  \]
\end{lemma}

\begin{proof}
    By induction on $\tau_1$.
\end{proof}

\section{Soundness in \texorpdfstring{\Sot}{S12}}\label{sec:SoundnessInS12}

We are now in the position to prove a form of soundness of pure equational reasoning in \Sot.
As a reminder, $\bvar(\D)$ denotes the set of variables occurring in~$\D$ that are bound by an application of substitution, see Definition~\ref{def:derivations}.

\begin{lemma}\label{lem:PhiRho}
  Let $\D$ be a derivation in Variable Normal Form,
  $\rho$ an assignment such that $\dom(\rho)$ and $\bvar(\D)$ are
  disjoint.
  Let $\tau$ be $\InstR_\D$ or $\InstL_\D$, and let
  $\Seq{F,\si',\rho'}$ be $\Phi(\tau, \Seq{F,\si,\rho})$.
  Then $\rho' = \rho$.
\end{lemma}

\begin{proof}  
  By induction on $\D$.
  We only consider the case for $\InstR_\D$, the case of
  $\InstL_\D$ will be similar.
  The only rule which changes $\rho$ is an application of 
  Substitution.
  In this case, $\D$ will be of the form
  \begin{prooftree}
    \AxiomC{$\cD_1$}
    \noLine
    \UnaryInfC{$t=u$}
    \UnaryInfC{$t[s/x]=u[s/x]$}
  \end{prooftree}
  and $\InstR_\D$ has the form
  \[
    \ISup[t,s/x] \seqadd \InstR_{\D_1}
    \seqadd \ISdwn[u,s/x] 
  \]
  By assumption we obtain $x\notin\dom(\rho)$ as $x\in\bvar(\D)$.
  The i.h.\ shows that the evaluation of $\InstR_{\D_1}$ does not
  change the assignment.
  Evaluating $\ISup[t,s/x]$ changes $\rho$ by mapping $x$ to some
  value,
  while evaluating $\ISdwn[u,s/x]$ removes $x$ from the domain of the
  assignment.
  Hence, the resulting overall assignment will be $\rho$ again.
\end{proof}

\begin{theorem}[\Sot]\label{thm:soundness in S12}
  Let $\D$ be a derivation of $t = u$ in Variable Normal Form.
  Let $\rho$ be an assignment with $\dom(\rho)\subseteq\var(t,u)$,
  and $F$ a model for \Ax.
  Let $\si_1, \si_2$ be given by
  \begin{align*}
    \Seq{F, \si_1, \rho} \quad&=\quad \Phi(\InstR_\D, \Seq{F, \Seq{}, \rho})\\
    \Seq{F, \si_2, \rho} \quad&=\quad \Phi(\InstL_\D, \Seq{F, \Seq{}, \rho})   
  \end{align*}
    Then
    \begin{align*}
        \evalfr{t} \quad&\sqsubseteq\quad \eval{u}_{F\updcons\si_1, \rho}\\
        \evalfr{u} \quad&\sqsubseteq\quad \eval{t}_{F\updcons\si_2, \rho}
    \end{align*}
\end{theorem}

Instead of proving the theorem directly, we prove the following stronger claim.

\begin{claim}[\Sot]\label{claim:soundness in S12}
  Fix some derivation $\cD$ in Variable Normal Form, 
  some model $F$ for~\Ax,
  and some integer $U$ such that $\G(F)+\lh(\cD)\le U$.
  Let 
  $\kappa=\G(F)$, and $X=\var(\D)$.

  Let $\D_0 \vdash t = u$ be a sub-derivation of $\D$.
  Let $\rho$ be an assignment,
  and $\si$ a sequence of updates based on $F,U$ and $\D$
  such that
  \begin{align*}
    \dom(\rho)\ &\subseteq\ X\setminus\bvar(\D_0) \\
    \G(\rho), \E(\si), \G(\si), \seqlh(\si) \ &\le\ 
    U - \lh(\D_0)
  \end{align*}
  Let $\si_1, \si_2$ be given by
  \begin{align*}
    \Seq{F, \si_1, \rho} \quad&=\quad \Phi(\InstR_{\D_0}, \Seq{F, \si, \rho})\\
    \Seq{F, \si_2, \rho} \quad&=\quad \Phi(\InstL_{\D_0}, \Seq{F, \si, \rho})   
  \end{align*}
  Then
  \begin{align*}
    \evalfr{t} \quad&\sqsubseteq\quad \eval{u}_{F\updcons\si_1, \rho}\\
    \evalfr{u} \quad&\sqsubseteq\quad \eval{t}_{F\updcons\si_2, \rho}
                      \enspace .
  \end{align*}
\end{claim}

Theorem~\ref{thm:soundness in S12} follows from Claim~\ref{claim:soundness in S12} by letting $\cD_0=\cD$,
$\rho$ as given,
$\si=\Seq{}$,
and $U=\max\{\G(F),\G(\rho)\}+\lh(\cD)$,

\begin{proof}[Proof of Claim~\ref{claim:soundness in S12}]
  We argue in \Sot.
  Let $\cD$, $F$, 
  $\kappa$, and $X$ be given as in the Claim. 
  We prove that for any $\cD_0$, $\rho$, $\si$, $\si_1$ and $\si_2$
  satisfying the conditions of the Claim,
  the assertion of the Claim holds, by induction on $\lh(\cD_0)$.
  Thus this is proven by logarithmic induction (LIND) on a $\pib1$-property,
  which is available in \Sot by Theorem~\ref{thm:Pi-LIND}.

  We consider cases according to the last rule applied
  in~$\cD_0$.
  The details for each case follow the same lines as in the proof of
  Claim~\ref{claim:soundness in S22}, except that now $\si_1$ and
  $\si_2$ are not chosen but given by the $\Phi$-function applied to
  sequences of instances that are extracted from derivations.
\end{proof}


\begin{corollary}\label{cor:Sot}
  The consistency of $\PETS(\Ax)$ is provable in $\Sot$.
\end{corollary}

\begin{proof}
  We argue in \Sot.
  Assume $\D$ is a $\PETS(\Ax)$ derivation ending in $0=1$.
  Using Proposition~\ref{prop:VNF} we can assume that $\D$ is in
  Variable Normal Form.
  Let $\rho$ be the empty assignment, and $F$ the empty model
  for~$\Ax$.
  Let $\si_1$ be given by
  \[
    \Seq{F, \si_1, \rho} \quad=\quad \Phi(\InstR_\D, \Seq{F, \Seq{}, \rho})
  \]
  By the previous Theorem~\ref{thm:soundness in S12}, we obtain
  \[
    0 \quad=\quad \evalfr{0}
    \quad\apprx\quad  \eval{1}_{F\updcons \sigma_1, \rho}
    \quad=\quad 1
  \]
  which is impossible.
\end{proof}

\bibliography{ref}

\end{document}